\documentclass[10pt,a4paper]{article}

\usepackage{amsfonts}
\usepackage{epsfig}
\usepackage{graphicx}
\usepackage{amsmath}
\usepackage{amssymb}

\def\NN{\mathbb N}
\def\ZZ{\mathbb Z}

\def\vep{\varepsilon}

\def\tr{\hbox{tr}}

\def\ms{\medskip}

\def\beq{\begin{equation}}
\def\eeq{\end{equation}}

\newtheorem{theo}{\bf Theorem}[section]
\newtheorem{prop}[section]{\bf Proposition}
\newtheorem{defi}{\bf Definition}[section]
\newtheorem{cor}{\bf Corollary}[section]

\newtheorem{lem}{\bf Lemma}[section]

\newcommand{\blanksquare}{\,\,\,$\sqcup\!\!\!\!\sqcap$}
\newenvironment{proof}{{\flushleft{\bf Proof: }}}{\hspace{\stretch{1}} \blanksquare\\}

\newcounter{example}
{{\stepcounter{example}}{\flushleft {\bf Example \arabic{example}:}}}%
{\par}

\begin{document}

\title{\textbf{Rationality of the zeta function for Ruelle-expanding maps}}

\author{M\'ario Alexandre Magalh\~aes} \thanks{Supported by a phd scholarship from Funda\c c\~ao para a Ci\^encia e a Tecnologia}

\maketitle

\begin{abstract}
We will prove that the zeta function for {\it Ruelle-expanding} maps is rational.
\end{abstract}

\noindent\emph{MSC 2000:} primary 37H15, 37D08; secondary 47B80.\\
\emph{keywords:} Zeta function; Ruelle-expanding\\

\begin{section}{\sc Introduction}
For any map $f$ with a finite number of periodic points for each period we can associate its {\it zeta function} defined as
\[\zeta_f(t)=\exp\left(\sum_{n=1}^\infty\frac{N_n(f)}{n}t^n\right)\]
where $N_n(f)$ is the number of periodic points with period $n$. In some cases, it is known that $\zeta_f(t)$ is a rational function. Those cases include the Markov subshifts of finite type (unilateral and bilateral) and Axiom A diffeomorphisms. Besides, in the case of the subshifts, an explicit formula relates the topological entropy and the radius of convergence of the zeta function. Another class of maps with this property is the {\it Ruelle-expanding} maps. This concept, created by Ruelle, generalizes the notion of expanding maps defined on manifolds, freeing its essence from the derivative's constraints. Our main result will be the following

\begin{theo}
If $f$ is Ruelle-expanding, then its zeta function is rational.
\end{theo}
Its proof will emulate the classical argument used to ensure the rationality of the zeta function for a $\mathcal{C}^1$ diffeomorphism defined on a hyperbolic set with local product structure, which profits by the existence of Markov partitions with arbitrarily small diameter. Within the Ruelle-expanding setting, we will prove the existence of a finite cover with analogous properties, which will play the same role the Markov partition did. Moreover, we will see that there is also a relation between the topological entropy and the radius of convergence of the zeta function in this case.
\end{section}

\begin{section}{\sc The zeta function}
\begin{defi}
\em If $f$ is a continuous map of a topological space $X$, let $N_n(f)$ denote the number of periodic points with period $n$, that is, the points $x$ for which $f^n(x)=x$. If $N_n(f)<\infty,\forall n\in\NN$, we define the {\it \bf zeta function} of $f$ as

\[\zeta_f(t)=\exp\left(\sum_{n=1}^\infty\frac{N_n(f)}{n}t^n\right)\]
\end{defi}

\ms

As the exponential is an entire function, the radius of convergence of $\zeta_f(t)$ is given by

\[\rho=\frac{1}{\limsup \sqrt[n]{\frac{N_n(f)}{n}}}=\frac{1}{\limsup \sqrt[n]{N_n(f)}}\] (since $\lim\sqrt[n]{n}=1$).

Let $L=-\log\rho$, so that $\rho=e^{-L}$. Then, we have

\[L=-\log \frac{1}{\limsup \sqrt[n]{N_n(f)}}=\limsup(1/n)\log N_n(f)\]

\begin{subsection}{\bf Examples}

\begin{subsubsection}{\emph{\sc Markov subshifts of finite type}}

Let $k$ be a natural number and $[k]$ the set $\left\{1,2,\ldots,k\right\}$ with the discrete topology. Consider $\Sigma(k)$ the product space $[k]^\ZZ$, whose elements are the sequences $\underline{a}=(\ldots,a_{-1},a_0,a_1,\ldots)$, with $a_n\in [k],\forall n\in\ZZ$. This space has a product topology, which can be generated by the metric given by

\[d(\underline{a},\underline{b})=\sum_{n=-\infty}^{\infty}\frac{\delta_n(\underline{a},\underline{b})}{2^{2\left|n\right|}}\]
where $\delta_n(\underline{a},\underline{b})$ is $0$ when $a_n=b_n$ and $1$ otherwise. Notice that

\[0\leq d(\underline{a},\underline{b})\leq\sum_{n\in\ZZ}\frac{1}{2^{2\left|n\right|}}=1+2\sum_{n\in\NN}\frac{1}{2^{2n}}=\frac{5}{3}\]
and that $d(\underline{a},\underline{b})\geq 1\Leftrightarrow a_0\neq b_0$, since

\[\sum_{n\in\ZZ\backslash\left\{0\right\}}\frac{1}{2^{2\left|n\right|}}=2\sum_{n\in\NN}\frac{1}{2^{2n}}=\frac{2}{3}<1\]

On $\Sigma(k)$ we have defined a homeomorphism, called {\it shift}, by

\[(\sigma(\underline{a}))_i=a_{i+1},i\in\ZZ\]
This way, $\sigma$ has a special class of closed invariant sets. Let $M_k$ be the set of $k\times k$ matrices with entries 0 or 1. For each $A\in M_k$, we define $\Sigma_A=\{\underline{a}\in\Sigma(k):A_{a_i a_{i+1}}=1\}$, which is a closed invariant subspace of $\Sigma(k)$. The pair $(\Sigma_A,\sigma_A)$, where $\sigma_A=\sigma|_{\Sigma_A}$, is called a {\it subshift of finite type}.

A matrix $A\in M_k$ is said to be {\it irreducible} if $\forall i,j\in[k],\exists n\in\NN:(A^n)_{i j}>0$. In this case, by the {\it Perron-Frobenius Theorem}, we know that it has a non-negative simple eigenvalue $\lambda$ which is greater than the absolute value of all the others eigenvalues, that is, such that $max_{i\in[k]}\left|\lambda_i\right|=\lambda$, where $\lambda_1,\lambda_2,\ldots,\lambda_k$ are all the eigenvalues of $A$. Besides, its entropy is $\log\lambda$. In particular, the entropy of the full shift $\sigma:\Sigma(k)\rightarrow\Sigma(k)$ is $\log k$. (See \cite{Walt}). For such a $\sigma_A$, we can actually compute the zeta function: it is a rational function and $L$ is precisely the entropy of $f$. Let us recall why.\\

We say that a finite sequence $a_0 a_1 ... a_n$ of elements in $[k]$ is {\it admissible} if $A_{a_i a_{i+1}}=1$. Let $N_n(p,q,A)$ denote the number of admissible sequences of length $n+1$ which start at $p$ and end at $q$.

\begin{prop}
$N_n(p,q,A)=(A^n)_{p q}$
\end{prop}

\begin{proof}
We use induction over $n$. For $n=1$, this is true by definition of $A$.
Suppose this is true for $n=m-1$. Then, for $n=m$ we have

\[N_m(p,q,A)=\sum_{r=1}^k N_{m-1}(p,r,A)A_{r q}=\sum_{r=1}^k (A^{m-1})_{p r}A_{r q}=(A^m)_{p q}\]
and the number of admissible sequences of length $n+1$ which start and end with the same element of $[k]$ is

\[\sum_{p=1}^k N_n(p,p,A)=\sum_{p=1}^k (A^n)_{p p}=\tr(A^n)\]
Notice that $\underline{a}\in\Sigma_A$ is a fixed point of $\sigma_A^n$ if and only if $a_i=a_{i+n},\forall i\in\ZZ$. Then, for each fixed point of $\sigma_A^n$ given by

\[\underline{a}=(...,a_0,a_1,a_2,...,a_n,a_{n+1},a_{n+2},...)=(...,a_0,a_1,a_2,...,a_0,a_1,a_2,...)\]
we can associate a unique admissible sequence of length $n+1$ given by $a_0 a_1 a_2 ... a_{n-1} a_0$. Therefore, the number of fixed points of $\sigma_A^n$ is $N_n(\sigma_A)=\tr(A^n)$.
\end{proof}

\begin{theo}
$\zeta_{\sigma_A}(t)=1/\det(I-tA)$
\end{theo}

\begin{proof}
Let $\lambda_1,\lambda_2,...,\lambda_k$ be the eigenvalues of $A$, so that

\[\det(tI-A)=(t-\lambda_1)(t-\lambda_2)\ldots(t-\lambda_k)\]
Replacing $t$ by $t^{-1}$, we get

\[\det(t^{-1}I-A)=(t^{-1}-\lambda_1)(t^{-1}-\lambda_2)\ldots(t^{-1}-\lambda_k)\]
and, multiplying both sides by $t^k$, we get

\[t^k\det(t^{-1}I-A)=t^k(t^{-1}-\lambda_1)(t^{-1}-\lambda_2)\ldots(t^{-1}-\lambda_k)\]
\[\det(I-tA)=(1-\lambda_1 t)(1-\lambda_2 t)\ldots(1-\lambda_k t)\]
Besides, we have

\[\zeta_{\sigma_A}(t)=\exp\left(\sum_{n=1}^\infty\frac{N_n(\sigma_A)}{n}t^n\right)=\exp\left(\sum_{n=1}^\infty\frac{\tr(A^n)}{n}t^n\right)\]
Since the eigenvalues of $A^n$ are $\lambda_1^n,\lambda_2^n,...,\lambda_k^n$, we get $\tr (A^n)=\sum_{m=1}^k \lambda_m^n$. So,

\[\zeta_{\sigma_A}(t)=\exp\left(\sum_{n=1}^\infty\frac{\sum_{m=1}^k \lambda_m^n}{n}t^n\right)=\exp\left(\sum_{m=1}^k\left(\sum_{n=1}^\infty\frac{(\lambda_m t)^n}{n}\right)\right)\]
Moreover, since $\sum_{n=1}^\infty\frac{t^n}{n}=\log\left(\frac{1}{1-t}\right)$, we have

\[\zeta_{\sigma_A}(t)=\exp\left(\sum_{m=1}^k\log\left(\frac{1}{1-\lambda_m t}\right)\right)=\]
\[=\exp\left(\log\left(\prod_{m=1}^k\left(\frac{1}{1-\lambda_m t}\right)\right)\right)=\frac{1}{\prod_{m=1}^k(1-\lambda_m t)}=\frac{1}{\det(I-tA)}\]
\end{proof}

\paragraph{}

For instance, if $A=\left(\begin{array}{cc}1&1\\1&0\end{array}\right)$, its eigenvalues are $\lambda_1=\frac{1+\sqrt{5}}{2}$ and $\lambda_2=\frac{1-\sqrt{5}}{2}$, so

\[\zeta_{\sigma_A}(t)=\frac{1}{(1-\lambda_1 t)(1-\lambda_2 t)}=\frac{1}{1-t-t^2}\]

\paragraph{}

\begin{prop}
Let $A$ be an irreducible matrix with entries 0 or 1. Then the topological entropy of $\sigma_A$ is $-\log\rho$, where $\rho$ is the radius of convergence of $\zeta_{\sigma_A}$.
\end{prop}

\begin{proof}
In fact, since $\zeta_{\sigma_A}(t)=1/\det(I-tA)$ and

\[\det(I-tA)=0\Leftrightarrow\prod_{m=1}^k(1-\lambda_m t)=0\Leftrightarrow \exists m\in[k]:t=1/\lambda_m \wedge \lambda_m\neq 0\]
the radius of convergence of $\zeta_{\lambda_A}$ is

\[\rho=\min\left\{\left|1/\lambda_i\right|: i\in [k] \wedge \lambda_i\neq 0\right\}=1/\max\left\{\left|\lambda_i\right|: i\in [k] \wedge \lambda_i\neq 0\right\}=1/\lambda\]
Then $L=-\log\rho=\log\lambda$, and this value is precisely the topological entropy of $\sigma_A$.
\end{proof}

\noindent{\bf Remark}: However, there are closed invariant subsets of $\Sigma(k)$ for which the zeta function for the restriction of $\sigma$ to those sets is not rational. In fact:

\begin{itemize}

\item The set of rational functions defined in a neighborhood of zero of the form $\exp\left(\sum_{n=1}^\infty\frac{N_n}{n}t^n\right)$, with $N_n\in\ZZ,\forall n\in\NN$, is countable. In particular, the set of rational functions which are zeta functions for some restriction of $\sigma$ is countable.

\item There is a noncountable collection of closed invariant subsets of $\Sigma(k)$ such that the zeta function for the restriction of $\sigma$ to those sets is distinct from each other.

\end{itemize}

Therefore, there is a noncountable collection of closed invariant subsets of $\Sigma(k)$ such that the zeta function for the restriction of $\sigma$ to those sets is not rational. For example, let $k=2$ and $S\subseteq\Sigma(2)$ be the set whose elements are the sequences with only one '1' and the periodic sequences with at most one '1' in a minimal period. Then, $S$ is a closed invariant subset of $\Sigma(2)$. Also, the number of periodic points of period $n$ in $S$ is equal to the sum of the divisors of $n$, $\sigma(n)$, plus one, that is, $N_n(\sigma|_S)=\sigma(n)+1$ and hence,

\[\zeta_{\sigma|_S}(t)=\exp\left(\sum_{n=1}^\infty\frac{\sigma(n)+1}{n}t^n\right)=\exp\left(\sum_{n=1}^\infty\frac{\sigma(n)}{n}t^n+\sum_{n=1}^\infty\frac{t^n}{n}\right)=\]

\[=\exp\left(-\log(s(t))-\log(1-t)\right)=\frac{1}{(1-t)s(t)}\]
where $s(t)=1-t-t^2+t^5+t^7-t^{12}+t^{15}-\ldots$ is a power series with arbitrarily long sequences of coefficients equal to zero. Since $s(t)$ isn't rational, $\zeta_{\sigma|_S}$ is not rational as well.\\

\end{subsubsection}

\noindent\textbf{2.1.2}\hspace{.2in}\emph{\sc Expansive maps}

\begin{defi}

\em Let $(X,d)$ be a metric space and $f:X\rightarrow X$ a continuous map. We say that $\vep$ is an {\it expansive constant} for $f$ if

\[d(f^n(x),f^n(y))\leq\vep,\forall n\in\NN_0\Longrightarrow x=y\]
The map $f$ is called {\it expansive} if it has an {\it expansive constant}. If $f:X\rightarrow X$ is a homeomorphism, we say that $\vep$ is an expansive constant for $f$ (and $f$ is expansive) if

\[d(f^n(x),f^n(y))\leq\vep,\forall n\in\ZZ\Longrightarrow x=y\]

\end{defi}
This property ensures that the periodic points of $f$ of period $n$ are isolated and the sets $N_n(f)$ are finite, $\forall n\in\NN$ (see \cite{Walt}). Moreover

\begin{prop}
If $(X,d)$ is a compact metric space and $f:X\rightarrow X$ is expansive, then $\zeta_f$ has a positive radius of convergence.
\end{prop}

\begin{proof} Suppose that $f$ is a continuous map with expansive constant $\vep$. Let $U_1,\ldots,U_r$ be a cover of $X$ with $diam(U_i)\leq\vep,\forall i\in\left[r\right]$. For each $x\in X$, let $\phi(x)=(a_0,a_1,a_2,\ldots)$, with $a_n=\min\{i\in\left[r\right]:f^n(x)\in U_i\}$. We can see that $\phi(x)=\phi(y)\Rightarrow d(f^n(x),f^n(y))\leq\vep,\forall n\in\NN_0\Rightarrow x=y$, so $\phi$ is injective. Also, if $x$ is periodic with period $n$, then so is $\phi(x)$. Since the number of periodic points in $\left[r\right]^{\NN_0}$ with period $n$ is $r^n$, we have $N_n(f)\leq r^n$ and

\[L=\limsup(1/n)\log N_n(f)\leq\log r\Longrightarrow \rho\geq 1/r>0\]
If $f$ is a homeomorphism with expansive constant $\vep$, then the proof is similar (we associate to each point of $X$ an unique sequence in $\left[r\right]^{\ZZ}$, which is periodic if the point is periodic).
\end{proof}

\paragraph{}

Since, for each expansive map, there is some $r\in\NN$ such that $N_n(f)\leq r^n,\forall n\in\NN$, we may also deduce that

\begin{cor}
\[1-r\left|t\right|\leq\left|\zeta_f(t)\right|\leq \frac{1}{1-r\left|t\right|}\]
\end{cor}

\begin{proof}

\[\left|\zeta_f(t)\right|=
\left|\exp\left(\sum_{n=1}^\infty\frac{N_n(f)}{n}t^n\right)\right|=
\exp\left(\sum_{n=1}^\infty\frac{N_n(f)}{n}Re(t^n)\right)\leq\]

\[\leq\exp\left(\sum_{n=1}^\infty\frac{r^n}{n}\left|t^n\right|\right)=
\exp\left(\sum_{n=1}^\infty\frac{(r\left|t\right|)^n}{n}\right)=
\exp\left(\log\left(\frac{1}{1-r\left|t\right|}\right)\right)=
\frac{1}{1-r\left|t\right|}\]

and, similarly,

\[\left|\zeta_f(t)\right|=
\exp\left(\sum_{n=1}^\infty\frac{N_n(f)}{n}Re(t^n)\right)\geq
\exp\left(\sum_{n=1}^\infty\frac{r^n}{n}(-\left|t^n\right|)\right)=
1-r\left|t\right|\]

for all t with $\left|t\right|<1/r$ (recall that $\rho\geq 1/r$).

\end{proof}

\noindent\textbf{2.1.3}\hspace{.2in}{\emph{\sc Axiom A diffeomorphisms}}

\begin{defi}
\em Let $f$ be a $\mathcal{C}^1$ diffeomorphism defined on a manifold $M$. A subset $\Lambda\subseteq M$ is {\it hyperbolic} if it is compact, $f$-invariant ($f(\Lambda)=\Lambda$) and there is a decomposition $T_{\Lambda}M=E^s_{\Lambda}\oplus E^u_{\Lambda}$ such that
	
\[D_x f(E^s_x)=E^s_{f(x)},\forall x\in\Lambda\]
\[D_x f(E^u_x)=E^u_{f(x)},\forall x\in\Lambda\]
\[\exists c>0,\lambda\in\left]0,1\right[:\forall x\in\Lambda,\forall n\geq 0,\]

\begin{center}
$\left\|D_x f^n(v)\right\|\leq c\lambda^n\left\|v\right\|,\forall v\in E^s_x$ and $\left\|D_x f^{-n}(v)\right\|\leq c\lambda^n\left\|v\right\|,\forall v\in E^u_x$
\end{center}
\end{defi}

\paragraph{}

For each $x\in\Lambda$, these expanding and contracting subbundles are tangent to the stable and unstable submanifolds,

\[W^s(x)=\left\{ y\in M:d(f^n(x),f^n(y))\rightarrow 0\right\}\]
\[W^u(x)=\left\{ y\in M:d(f^{-n}(x),f^{-n}(y))\rightarrow 0\right\}\]
Besides, for small $\vep$, the local submanifolds

\[W^s_{\vep}(x)=\left\{ y\in M:d(f^n(x),f^n(y))<\vep,\forall n\geq 0\right\}\]
\[W^u_{\vep}(x)=\left\{ y\in M:d(f^{-n}(x),f^{-n}(y))<\vep,\forall n\geq 0\right\}\]
are $\mathcal{C}^1$ disks embedded in $M$ and there is $\delta>0$ such that, if the distance between two points $x$ and $y$ in $\Lambda$ is less then $\delta$, then $W^s_{\vep}(x)$ and $W^u_{\vep}(y)$ intersect transversely at an unique point, denoted by $[x,y]$.

In particular, if $y=x$ then $W^s_{\vep}(x)\cap W^u_{\vep}(x)=\{x\}$, which means that $\vep$ is an expansive constant for $f$ (see \cite{Shub}). We say that $\Lambda$ has a {\it local product structure} if $[x,y]\in\Lambda,\forall x, y\in\Lambda$.\\

If $f$ is a $\mathcal{C}^1$ diffeomorphism defined on a hyperbolic set with local product structure, then $f$ is expansive, so $N_n(f)<\infty,\forall n\in\NN$ and we can define the zeta function for $f$. And moreover, as proved in \cite{Shub},

\begin{theo}
The zeta function of a $\mathcal{C}^1$ diffeomorphism on a hyperbolic set with local product structure is rational.
\end{theo}

As a consequence, if $f$ is a $\mathcal{C}^1$ diffeomorphism such that $\overline{Per(f)}$ is hyperbolic, then $\zeta_f(t)$ is a rational function: in fact, it is known that, if $\overline{Per(f)}$ is hyperbolic, then it has a local product structure; and $\zeta_f(t)=\zeta_{f|_{\overline{Per(f)}}}(t)$. In particular, if $f$ is Axiom A ($\Omega(f)$ is hyperbolic and $\Omega(f)=\overline{Per(f)}$, where $\Omega(f)$ denotes the set of non-wandering points of $f$), then $\zeta_f(t)$ is rational.\\

The main ingredient of the classical argument to prove this theorem is the existence of a Markov partition of arbitrarily small diameter, which establishes a codification of most of the orbits of $f$ through a subshift of finite type (for which we already know how to count the periodic points), and a sharp way to translate the properties of the zeta function from the subshift to the diffeomorphism setting.

\end{subsection}
\end{section}

\begin{section}{\sc Ruelle-expanding maps}

Here, we will explain the nature of another class of maps, called {\it Ruelle-expanding}, whose zeta function is rational.

\begin{defi}

\em Let $(K,d)$ be a compact metric space and $f:K\rightarrow K$ a continuous map. We say that $f$ is {\it Ruelle-expanding} if there are $r>0$, $0<\lambda<1$ and $c>0$ such that:

\begin{itemize}

\item $\forall x,y \in K, x\neq y \wedge f(x)=f(y)\Longrightarrow d(x,y)>c$

\item $\forall x \in K,\forall a \in f^{-1}(\{x\}),\exists \phi:B_r(x)\rightarrow K$ with

$\phi(x)=a$

$(f\circ\phi)(y)=y,\forall y \in B_r(x)$

$d(\phi(y),\phi(z))\leq \lambda d(y,z),\forall y,z\in B_r(x)$

\end{itemize}

\end{defi}

\paragraph{}

{\bf Examples}

\begin{itemize}

\item Let $M$ be a compact manifold and $f:M\rightarrow M$ a $\mathcal{C}^1$ map. We say that $f$ is {\it expanding} if $\exists\lambda\in\left]0,1\right[:\forall x\in M,\left\|D_x f(v)\right\|\geq 1/\lambda \left\|v\right\|$. It can be proved (see \cite{MC}) that, in this particular case, this condition is equivalent to the previous two from the last definition. So, $f$ is expanding if and only if it is Ruelle-expanding.

One example of such a map is the application\\ $\begin{array}{rcl}f:S^1&\rightarrow&S^1\\z&\mapsto&z^k\end{array}$, with $k\in\ZZ$ and $k>1$

(it is easy to see that $f$ is expanding, with $\lambda=1/k$). Notice that, for this map, we have $N_n(f)=k^n-1$. So,

\[
\zeta_f(t)=\exp\left(\sum_{n=1}^\infty\frac{k^n-1}{n}t^n\right)=\exp\left(\sum_{n=1}^\infty\frac{(kt)^n}{n}-\sum_{n=1}^\infty\frac{t^n}{n}\right)=
\]

\[
=\exp\left(\log\left(\frac{1}{1-kt}\right)-\log\left(\frac{1}{1-t}\right)\right)=\exp\left(\log\left(\frac{1-t}{1-kt}\right)\right)=\frac{1-t}{1-kt}
\]

which is a rational function (with a pole at $\frac{1}{k}$, so $\rho=\frac{1}{k}=\exp(-\log(k))=\exp(-h(f))$).

\item Let $\Sigma(k)^+$ be the product space $[k]^{\NN_0}$, whose elements are the sequences $\underline{a}=(a_0,a_1,\ldots)$, with $a_n\in [k],\forall n\in\NN_0$.
As its bilateral version, this space has a product topology which can be generated by the metric given by

\[d(\underline{a},\underline{b})=\sum_{n=0}^{\infty}\frac{\delta_n(\underline{a},\underline{b})}{2^n}\]

where $\delta_n(\underline{a},\underline{b})$ is $0$ when $a_n=b_n$ and $1$ otherwise. The \textit{unilateral shift} is the map of $\Sigma(k)^+$ given by

\[(\sigma^+(\underline{a}))_i=a_{i+1},i\in\NN_0\]

For each $A\in M_k$, we define $\Sigma_A^+=\{\underline{a}\in\Sigma(k)^+:A_{a_i a_{i+1}}=1\}$. The pair $(\Sigma_A^+,\sigma_A^+)$, where $\sigma_A^+=\sigma^+|_{\Sigma_A^+}$, is called a {\it unilateral subshift of finite type}. If $A$ is irreducible, then it is easy to see that $\sigma_A^+$ is Ruelle-expanding, with $r=1$ and $\lambda=c=1/2$, since:\\

- If $\underline{a}\neq\underline{b}$ and $\sigma_A^+(\underline{a})=\sigma_A^+(\underline{b})$, then $a_0\neq b_0$, so $d(\underline{a},\underline{b})\geq 1>c$.\\

- If $r=1$, then, for any $\underline{a}\in\Sigma_A^+$ we have $B_r(\underline{a})=\{\underline{b}\in\Sigma_A^+:b_0=a_0\}$ since, as we have seen, $b_0\neq a_0\Rightarrow d(\underline{a},\underline{b})\geq 1=r$. Also, the pre-images of $\underline{a}=(a_0,a_1,a_2,\ldots)$ are of the form $(x,a_0,a_1,\ldots)$, where $A_{x a_0}=1$ (there is at least one $x\in [k]$ such that $A_{x a_0}=1$ because $A$ is irreducible). If we define $\phi(\underline{b})=(x,b_0,b_1,b_2,\ldots)$ for $\underline{b}=(b_0,b_1,b_2,\ldots)\in B_r(\underline{a})$ (that is, with $a_0=b_0$), then $\sigma_A^+(\phi(\underline{b}))=\underline{b}$ and

\[d(\phi(\underline{b}),\phi(\underline{c}))=\sum_{n=1}^{\infty}\frac{\delta_{n-1}(\underline{b},\underline{c})}{2^n}=\sum_{n=0}^{\infty}\frac{\delta_n(\underline{b},\underline{c})}{2^{n+1}}=\frac{d(\underline{b},\underline{c})}{2}=\lambda d(\underline{b},\underline{c}),\forall\underline{b},\underline{c}\in B_r(\underline{a})\]

\end{itemize}

If, to simplify the notation, we denote by $\sigma$ the map $\sigma_A^+$, then $\underline{a}\in\Sigma_A^+$ is a fixed point of $\sigma^n$ if and only if $a_i=a_{i+n},\forall i\in\NN_0$. For each fixed point of $\sigma^n$ given by

\[
\underline{a}=(a_0,a_1,a_2,...,a_n,a_{n+1},a_{n+2},...)=(a_0,a_1,a_2,...,a_0,a_1,a_2,...)
\]
we can associate a unique admissible sequence of length $n+1$ given by $a_0 a_1 a_2 ... a_{n-1} a_0$. So, the number of fixed points of $\sigma^n$ is $N_n(\sigma)=\tr(A^n)$ and $\zeta_{\sigma}(t)=1/\det(I-tA)$, also a rational function (with poles at the inverses of the eigenvalues of $A$).

\begin{defi}

\em Let $f:K\rightarrow K$ be Ruelle-expanding and $S\subseteq K$. Given $n\in\NN$, we say that $g:S\rightarrow K$ is a {\it contractive branch} of $f^{-n}$ if

\begin{itemize}

	\item $(f^n\circ g)(x)=x, \forall x \in S$

	\item $d((f^j\circ g)(x),(f^j\circ g)(y))\leq \lambda^{n-j}d(x,y),\forall x,y\in S,j\in\{0,1,\ldots,n\}$

\end{itemize}

\end{defi}

It is easy to see (details in \cite{MC}) that, given $x\in K$ and $a\in f^{-n}(\{x\})$ for some $n\in\NN$, there is always a contractive branch $g:B_r(x)\rightarrow K$ of $f^{-n}$ with $g(x)=a$. Moreover,

\begin{prop}
Let $B(n,\vep,x)=\{y\in K : d(f^j(x),f^j(y))<\vep,\forall j\in\{0,\ldots,n\}\}$. There is some $\vep_0<r$ such that, for every $\vep$ with $0<\vep<\vep_0$, we have
\begin{itemize}
\item $\forall n\in\NN$, $B(n,\vep,x)=g(B_{\vep}(f^n(x)))$, where $g:B_r(f^n(x))\rightarrow K$ is a contractive branch of $f^{-n}$ with $g(f^n(x))=x$
\item $\vep$ is an expansive constant for $f$
\end{itemize}
\end{prop}

\begin{proof} See \cite{MC}.
\end{proof}

\begin{prop}
$K=\bigcup_{n\geq 0}f^{-n}(\overline{Per(f)})$, where $Per(f)$ is the set of periodic points for $f$. In particular, $Per(f)\neq\emptyset$.
\end{prop}

\begin{proof} See \cite{MC}.
\end{proof}

Notice that, since $f$ is expansive, we can consider the zeta function for $f$ and, as $Per(f)\neq\emptyset$, given $x\in K$ with $f^k(x)=x$ for some $k\in\NN$, we have $f^{nk}(x)=x$ and $N_{nk}(f)\geq 1,\forall n\in\NN$, which implies that $L\geq\limsup\frac{1}{nk}\log N_{nk}(f)\geq 0$ and $\rho\leq 1$.\\

Is there any relation between $L$ and $h(f)$ if $f$ is Ruelle-expanding? In fact, we have that $L\leq h(f)$ but, to prove it, we need to simplify the calculus of $h(f)$. Let us first recall briefly how to evaluate, in general, this number.

Given a metric space $(X,d)$ and a uniformly continuous map $f:X\rightarrow X$, for every $n\in\NN$, we define a dynamical metric $d_n$ on $X$ by

\[d_n(x,y)=\max\{d(f^i(x),f^i(y)),i\in\{0,1,\ldots,n-1\}\}\]
and the corresponding open dynamical ball, with center $x$ and radius $r$,

\[B(n-1,r,x)=\{y\in K : d(f^j(x),f^j(y))< r,\forall j\in\{0,\ldots,n-1\}\}=\bigcap_{i=0}^{n-1}f^{-i}(B_r(f^i(x)))\]
and closed dynamical ball

\[\overline{B}(n-1,r,x)=\{y\in K : d(f^j(x),f^j(y))\leq r,\forall j\in\{0,\ldots,n-1\}\}=\bigcap_{i=0}^{n-1}f^{-i}(\overline{B}_r(f^i(x)))\]
Accordingly,

\begin{defi}
\em Let $n\in\NN$, $\vep>0$ and $K$ be a compact subset of $X$. Given a subset $F$ of $X$, we say that $F$ $(n,\vep)-$ {\it spans} $K$ with respect to $f$ if
\[\forall x\in K,\exists y\in F:d_n(x,y)\leq\vep\]
or, equivalently,
\[K\subseteq\bigcup_{y\in F}\overline{B}(n-1,\vep,y)\]
\end{defi}

\begin{defi}
\em Let $n\in\NN$, $\vep>0$ and $K$ be a compact subset of $X$. We define $r_n(\vep,K)$ as the smallest cardinality of any $(n,\vep)$ spanning set for $K$ with respect to $f$.
\end{defi}

Notice that, since $K$ is compact, we have $r_n(\vep,K)<\infty$; and $\vep_1<\vep_2\Longrightarrow r_n(\vep_1,K)\geq r_n(\vep_2,K)$.

\begin{defi}
\em Let $\vep>0$ and $K$ be a compact subset of $X$. Then
\[r(\vep,K)=r(\vep,K,f)=\limsup_{n\rightarrow\infty}(1/n)\log r_n(\vep,K)\]
\end{defi}

\begin{defi}
\em If for each compact subset $K$ of $X$ we denote by $h(f,K)$ the limit $\lim_{\vep\rightarrow 0} r(\vep,K,f)$, then the \textit{topological entropy} of $f$ is
$h(f)=\sup\{h(f,K),K$ compact subset of $X\}$.
\end{defi}

Sometimes it is useful to use an equivalent way of defining topological entropy which uses \emph{separated sets} instead of spanning ones.

\begin{defi}
\em Let $n\in\NN$, $\vep>0$ and $K$ be a compact subset of $X$. Given a subset $E$ of $K$, we say that \textit{$E$ is $(n,\vep)$ \textit{separated with respect to} $f$} if
\[\forall x,y\in E,d_n(x,y)\leq\vep \Longrightarrow x=y\]
or, equivalently,
\[\forall x\in E,\overline{B}(n-1,\vep,x)=\{x\}\]
\end{defi}

\begin{defi}
\em Let $n\in\NN$, $\vep>0$ and $K$ be a compact subset of $X$. We define $s_n(\vep,K)$ as the largest cardinality of any $(n,\vep)$ separated set for $K$ with respect to $f$.
\end{defi}

Observe that $r_n(\vep,K) \leq s_n(\vep,K) \leq r_n(\vep/2,K)$ and so, since $r_n(\vep/2,K)<\infty$, we have $s_n(\vep,K)<\infty$; besides, $\vep_1<\vep_2\Longrightarrow s_n(\vep_1,K)\geq s_n(\vep_2,K).$

\begin{defi}
\em Let $\vep>0$ and $K$ be a compact subset of $X$. We define
\[s(\vep,K)=s(\vep,K,f)=\limsup_{n\rightarrow\infty}(1/n)\log s_n(\vep,K)\]
\end{defi}

As a consequence of the previous inequalities, we get $r(\vep,K) \leq s(\vep,K) \leq r(\vep/2,K)$ and so (see \cite{Walt})\\

\begin{prop}

\begin{itemize}
\item[(a)] For any compact subset $K$ of $X$, we have $h(f,K)=\lim_{\vep\rightarrow 0} s(\vep,K).$
\item[(b)] $h(f)=\sup_K h(f,K)=\sup_K \lim_{\vep\rightarrow 0} s(\vep,K,f).$
\item[(c)] In case $X$ is compact, then

$h(f)=h(f,X)=\lim_{\vep\rightarrow 0}\limsup(1/n)\log r_n(\vep,X)=\lim_{\vep\rightarrow 0}\limsup(1/n)\log s_n(\vep,X)$.
\end{itemize}

\end{prop}

Let us now go back to Ruelle-expanding maps.

\begin{prop}
If $f:X\rightarrow X$ is a Ruelle-expanding map of a compact metric space $(X,d)$, then $h(f)=r(\vep_0,X)=s(\vep_0,X)$ for all $\vep_0<\vep/4$, where $\vep$ is an expansive constant for $f$.
\end{prop}

\begin{proof}
See \cite{Walt}. (Although the proof is for expansive homeomorphisms, it can be easily adapted for expansive maps.)
\end{proof}

\begin{cor}
For any Ruelle-expanding map we have $L\leq h(f)$, that is, the radius of convergence of the zeta function is $\rho\geq \exp(-h(f))$.
\end{cor}

\begin{proof}
Let $p$ and $q$ be periodic points of $f$, with $f^n(p)=p$ and $f^n(q)=q$ for some $n\in\NN$. Then, we have

\[d_n(p,q)\leq\vep_0\Longrightarrow d_n(p,q)\leq\vep\Longrightarrow d(f^i(p),f^i(q))\leq\vep,\forall i\in\{0,1,\ldots,n-1\}\Longrightarrow\] \[\Longrightarrow d(f^i(p),f^i(q))\leq\vep,\forall i\in\NN_0\Longrightarrow p=q\]

So, the set $P_n$ of periodic points $p$ with $f^n(p)=p$ is a $(n,\vep_0)$ separated set for $X$ and $s_n(\vep_0,X)\geq card(P_n)=N_n(f)$. Consequently,

\[L=\limsup(1/n)\log N_n(f)\leq\limsup(1/n)\log s_n(\vep_0,X)=s(\vep_0,X)=h(f)\]

\end{proof}

This yields a link between $h(f)$ and the number of pre-images of the points in $X$ by $f$.

\begin{lem}
If $(X,d)$ is a compact metric space and $f:X\rightarrow X$ is a Ruelle-expanding map, then there is a $k\in\NN$ such that $card(f^{-1}(\{x\}))\leq k,\forall x\in X$.
\end{lem}

\begin{proof}
If we set $E=f^{-1}(\{x\})$ then we have $f(u)=f(v)=x,\forall u,v\in E,u\neq v$, so $d_1(u,v)=d(u,v)>c$ and $E$ is a $(1,c)$ separated set. Since $card(E)\leq s_1(c,X)<\infty$, we can take $k=s_1(c,X)$.
\end{proof}

\begin{prop}
$h(f)\leq \log(k)$, with equality if $card(f^{-1}(\{x\}))=k,\forall x\in X$.
\end{prop}

\begin{proof}
Let $\vep_0<min\{\vep/4,c,r\}$. Since $X$ is compact, there is a finite set $F$ for which we can write

\[X=\bigcup_{y\in F}\overline{B}_{\vep_0}(y)\]
Given $x\in X$ and $n\in\NN$, let $y\in F$ be such that $d(f^n(x),y)\leq\vep_0$ and let $g:B_r(f^n(x))\rightarrow X$ be a contractive branch of $f^{-n}$ with $g(f^n(x))=x$. If we take $z=g(y)$, we have

\begin{itemize}

	\item $f^n(z)=f^n(g(y))=y\Longrightarrow z\in f^{-n}(F)$

	\item $d(f^i(x),f^i(z))=d(f^i(g(f^n(x))),f^i(g(y)))\leq\lambda^{n-i}d(f^n(x),y)\leq\lambda^{n-i}\vep_0\leq\vep_0,\forall i\in\{0,1,\ldots,n-1\}\Longrightarrow d_n(x,z)\leq\vep_0$

\end{itemize}
So, $f^{-n}(F)$ is a $(n,\vep_0)$ spanning set for $X$. Therefore, $r_n(\vep_0,X)\leq card(f^{-n}(F))\leq k^n card(F),\forall n\in\NN$ and we get

\[h(f)=r(\vep_0,X)=\limsup(1/n)\log r_n(\vep_0,X)\leq\limsup(1/n)\log(k^n card(F))=\]
\[=\limsup(\log k+(1/n)\log(card(F)))=\log k\]

As a consequence, we have $0\leq L\leq \log k$ and $1/k\leq\rho\leq 1$.\\

Suppose now that there is some $k\in\NN$ such that $card(f^{-1}(\{x\}))=k,\forall x\in X$. Take a point $x\in X$. If we consider $E_n=f^{-n}(\{x\})$, then we have $f^n(u)=f^n(v)=x,\forall u,v\in E_n,u\neq v$. If $f(u)=f(v)$, then $d_n(u,v)\geq d(u,v)>c$, otherwise, we have $f(u)\neq f(v)$. Admitting the last case, if $f^2(u)=f^2(v)$, then $d_n(u,v)\geq d(f(u),f(v))>c$, otherwise, we have $f^2(u)\neq f^2(v)$. Proceeding, and since we have $f^n(u)=f^n(v)$, there must be some $j\in\{1,\ldots,n\}$ for which $f^j(u)=f^j(v)$ and $f^{j-1}(u)\neq f^{j-1}(v)$, so $d_n(u,v)\geq d(f^{j-1}(u),f^{j-1}(v))>c$ and $E_n$ is a $(n,c)$ separated set. Since $card(E_n)=k^n$, we have $k^n\leq s_n(c,X)\leq s_n(\vep_0,X)$ and we get

\[h(f)=s(\vep_0,X)=\limsup(1/n)\log s_n(\vep_0,X)\geq\limsup(1/n)\log(k^n)=\log k\]
which allow us to conclude that, in this particular case, $h(f)=\log k$.

\end{proof}

Now, our goal will be to prove the rationality of the zeta function for Ruelle-expanding maps. Recall that the existence of a Markov partition was an essential ingredient in the proof of the rationality of the zeta function for $\mathcal{C}^1$ diffeomorphisms defined on a hyperbolic set with local product structure. In the case of Ruelle-expanding maps, we will prove the existence of a finite cover with analogous properties, which will play the same role the Markov partition did.

\begin{prop}
Let $f$ be a Ruelle-expanding map defined on a compact set $K$. Let $\vep$ denote an expansive constant for $f$. Then, $K$ has a finite cover $\{R_1,...,R_n\}$ with the following properties:

\begin{itemize}

\item Each $R_i$ has a diameter less than $min\{\vep,c/2\}$ and is proper, that is, equal to the closure of its interior.

\item $\stackrel{\circ}{R_i}\cap\stackrel{\circ}{R_j}=\emptyset,\forall i,j\in[n],i\neq j$

\item $f(\stackrel{\circ}{R_i})\cap\stackrel{\circ}{R_j}\neq\emptyset\Longrightarrow \stackrel{\circ}{R_j}\subseteq f(\stackrel{\circ}{R_i})$

\end{itemize}

\end{prop}

\noindent\textbf{Remark}: If $\stackrel{\circ}{R_j}\subseteq f(\stackrel{\circ}{R_i})$, then $R_j=\overline{\stackrel{\circ}{R_j}}\subseteq \overline{f(\stackrel{\circ}{R_i})}\subseteq f\left(\overline{\stackrel{\circ}{R_i}}\right)=f(R_i)$ and the last condition means that $f(\stackrel{\circ}{R_i})\cap\stackrel{\circ}{R_j}\neq\emptyset\Longrightarrow R_j\subseteq f(R_i)$

\paragraph{}

To prove this proposition, we will begin by a shadowing lemma.

\begin{lem}
Let $f:K\rightarrow K$ be Ruelle-expanding. For any $\beta\in]0,r[$ there is some $\alpha>0$ such that, if $(x_n)_{n\in\NN_0}$ is a $\alpha$-pseudo orbit in $K$ \rm(\it that is, if $d(f(x_n),x_{n+1})<\alpha,\forall n\in\NN_0$\rm)\it, then it admits a $\beta$-shadow \rm(\it that is, a point $x\in K$ such that $d(f^n(x),x_n)<\beta,\forall n\in\NN_0$\rm)\it. Besides, the $\beta$-shadow is unique if $\beta<\vep/2$, where $\vep$ is an expansive constant for $f$.
\end{lem}

\begin{proof}
We will start proving this assertion for finite $\alpha$-pseudo orbits. Let $\beta\in]0,r[$ and $(x_0,x_1,\ldots,x_n)$ be such that $d(f(x_{k-1}),x_k)<\alpha,\forall k\in[n]$ for some $\alpha>0$. If $y_n=x_n$, then $d(y_n,x_n)=0<\beta$. Now, suppose that $d(y_k,x_k)<\beta$ for $k\in[n]$. Since $d(f(x_{k-1}),x_k)<\alpha$, we have $d(y_k,f(x_{k-1}))<\alpha+\beta<r$ if we assume $\alpha<r-\beta$. Then, we can take $y_{k-1}=g(y_k)$, where $g:B_r(f(x_{k-1}))\rightarrow K$ is a contractive branch of $f^{-1}$ with $g(f(x_{k-1}))=x_{k-1}$, and we have $d(y_{k-1},x_{k-1})\leq \lambda d(y_k,f(x_{k-1}))<\lambda (\alpha+\beta)<\beta$ if we assume $\alpha<\frac{1-\lambda}{\lambda}\beta$. Also, notice that $y_k=f(y_{k-1}),\forall k\in[n]$, so that $y_k=f^k(x),\forall k\in[n]$ for $x=y_0$. Hence, it suffices to take $\alpha<\min\{r-\beta,\frac{1-\lambda}{\lambda}\beta\}$.

Now, take $\beta\in]0,r[$ and let $(x_n)_{n\in\NN_0}$ be a $\alpha$-pseudo orbit, with $\alpha<\min\{r-\beta/2,\frac{1-\lambda}{\lambda}\beta/2\}$. Let $z_n$ be a $\beta/2$-shadow of $(x_0,x_1,\ldots,x_n)$; since $K$ is compact, there is some subsequence $(z_{n_k})_k$ converging to some point $z\in K$. We have $d(f^i(z_{n_k}),x_i)<\beta/2,\forall i\in\{0,1,\ldots,n_k\}$, so, for $i\in\NN_0$ fixed we get $d(f^i(z),x_i)=\lim d(f^i(z_{n_k}),x_i)\leq\beta/2<\beta$ and we conclude that $z$ is a $\beta$-shadow of $(x_n)_{n\in\NN_0}$.

For the uniqueness of the $\beta$-shadow when $\beta<\vep/2$, suppose that $z$ and $z'$ are both $\beta$-shadows of $(x_n)_{n\in\NN_0}$. Then, we have $d(f^i(z),f^i(z'))\leq d(f^i(z),x_i)+d(f^i(z'),x_i)<2\beta<\vep,\forall i\in\NN_0$, so $z=z'$.

\end{proof}
\paragraph{}

Let $\vep$ be an expansive constant for $f$ with $\vep<r$ and fix some $\beta<\min\{\vep/2,c/4\}$. Let $\alpha$ be given by the previous lemma and $\gamma\in]0,\alpha/2[$ be such that $d(x,y)<\gamma\Rightarrow d(f(x),f(y))<\alpha/2,\forall x,y\in K$. Since $K$ is compact, we can take $\{p_1,\ldots,p_k\}$ such that $K=\bigcup_{i=1}^k B_{\gamma}(p_i)$. We define a matrix $A\in M_k$ by

\begin{center}
$A_{ij}=1$ if $d(f(p_i),p_j)<\alpha$ and $A_{ij}=0$ otherwise.\\
\end{center}

For every $\underline{a}\in\Sigma^+_A$ the sequence $(p_{a_i})_{i\in\NN_0}$ is a $\alpha$-pseudo orbit, so it admits an unique $\beta$-shadow which we will denote by $\theta(\underline{a})$. Therefore, we have defined a map $\theta:\Sigma_A^+\rightarrow K$.

\begin{lem}
$\theta$ is a semiconjugacy of $\sigma^+_A$ and $f$, that is, $\theta$ is surjective, continuous and verifies $f\circ\theta=\theta\circ\sigma_A^+$.
\end{lem}

\begin{proof}
Given $x\in K$, we can take $a_i\in [k]$ so that $d(f^i(x),p_{a_i})<\gamma$ for any $i\in\NN_0$; then, $d(f(p_{a_i}),p_{a_{i+1}})\leq d(f(p_{a_i}),f(f^i(x)))+d(f^{i+1}(x),p_{a_{i+1}})<\alpha
/2+\gamma<\alpha$ and $(p_{a_i})_{i\in\NN_0}$ is a $\alpha$-pseudo orbit. So, $x=\theta(\underline{a})$ and $\theta$ is surjective.

For the continuity, since $K$ is compact it suffices to see that, for any two sequences $(\underline{s}^n)_{n\in\NN}$ and $(\underline{t}^n)_{n\in\NN}$ converging to the same limit $l$ in $\Sigma_A^+$ whose images under $\theta$ converge respectively to $s$ and $t$ in $K$, we have $s=t$.
Fix some $i\in\NN_0$; for any $n\in\NN$, we have $d(f^i(\theta(\underline{s}^n)),p_{s^n_i})<\beta$ and $d(f^i(\theta(\underline{t}^n)),p_{t^n_i})<\beta$. So, taking limits we have $d(f^i(s),p_{l_i})\leq\beta$ and $d(f^i(t),p_{l_i})\leq\beta$. Hence, $d(f^i(s),f^i(t))\leq 2\beta<\vep$ and, since $\vep$ is an expansive constant for $f$, we get $s=t$.

Finally, the relation $f\circ\theta=\theta\circ\sigma^+_A$ is a consequence of the unicity of the $\beta$-shadow and the fact that, if $x$ is a $\beta$-shadow for $(p_{a_i})_i$, then $f(x)$ is a $\beta$-shadow for $(p_{a_{i+1}})_i=(p_{\sigma_A^+(a_i)})_i$.

\end{proof}
\paragraph{}

Let $T_i=\{\theta(\underline{a}):a_0=i\}$ for $i\in [k]$. Then, $T_i=\theta(C_i)$ where $C_i=\{\underline{a}\in\Sigma_A^+:a_0=i\}$ and, since $\Sigma_A^+=\bigcup_{i=1}^k C_i$, we have $K=\bigcup_{i=1}^k T_i$ because $\theta$ is surjective. Hence, $\{T_i,i\in[k]\}$ is a finite closed cover of $K$ ($T_i$ is closed since $C_i$ is compact and $\theta$ is continuous).

\begin{lem}
If $A_{ij}=1$, then $T_j\subseteq f(T_i)$ and $\stackrel{\circ}{T_j}\subseteq f(\stackrel{\circ}{T_i})$. Also, given $x\in T_i$ with $f(x)\in T_j$, if $g:B_r(f(x))\rightarrow K$ is a contractive branch of $f^{-1}$ with $g(f(x))=x$, then $g(T_j)\subseteq T_i$ and $g(\stackrel{\circ}{T_j})\subseteq \stackrel{\circ}{T_i}$.
\end{lem}

\begin{proof}
Given any $y\in T_j$, we have $y=\theta(\underline{b})$ for some $\underline{b}\in\Sigma_A^+$ with $b_0=j$. Since $A_{ij}=1$, we can take $\underline{c}=(i,b_0,b_1,b_2,\ldots)\in \Sigma_A^+$, and so $y=\theta(\underline{b})=\theta(\sigma_A^+(\underline{c}))=f(\theta(\underline{c}))\in f(\theta(C_i))=f(T_i)$. Then, $T_j\subseteq f(T_i)$\\

Notice that $T_j\subseteq B_\beta(p_j)$. Since $d(f(x),p_j)<\beta$, we have $T_j\subseteq B_{2\beta}(f(x))\subseteq B_r(f(x))$. Let $g:B_r(f(x))\rightarrow K$ be a contractive branch of $f^{-1}$ with $g(f(x))=x$. Given $y\in T_j$, we have $y=f(z)$ for some $z\in T_i$. Then,

\[
d(g(y),z)\leq d(g(y),g(f(x)))+d(x,p_i)+d(p_i,z)<d(y,f(x))+2\beta<4\beta<c
\]
and, since $f(g(y))=y=f(z)$, we get $g(y)=z\in T_i$. So, $g(T_j)\subseteq T_i$.\\

It is easy to see that $g:B_r(f(x))\rightarrow g(B_r(f(x)))$ is a homeomorphism, with $g^{-1}=f|_{g(B_r(f(x)))}:g(B_r(f(x)))\rightarrow B_r(f(x))$. Therefore, we conclude that $g(\stackrel{\circ}{T_j})=\overbrace{g(T_j)}^{\circ}\subseteq\stackrel{\circ}{T_i}$ and $\stackrel{\circ}{T_j}=f(g(\stackrel{\circ}{T_j}))\subseteq f(\stackrel{\circ}{T_i})$.

\end{proof}
\paragraph{}

Let $Z=K\backslash\bigcup_{i=1}^k \partial T_i$. Notice that, since $T_i$ is a closed set, $\partial T_i$ has empty interior. So, $Z$ is dense in $K$. Given $x\in Z$ we define

\begin{center}
$T_i^*(x)=\stackrel{\circ}{T_i}$ if $x\in\stackrel{\circ}{T_i}$\\   $T_i^*(x)=K\backslash T_i$ if, otherwise, $x\notin T_i$\\   $R(x)=\bigcap_{i=1}^k T_i^*(x)$
\end{center}

The sets $R(x)$ satisfy the following properties:

\begin{itemize}

\item  $R(x)$ is open (because it is a finite intersection of open sets)

\item  $x\in R(x)$ (because $x\in T_i^*(x),\forall i\in[k]$)

\item  $R(x)\subseteq\stackrel{\circ}{T_i}$ for some $i\in[k]$

(since $\bigcap_{i=1}^k K\backslash T_i=K\backslash\bigcup_{i=1}^k T_i=\emptyset$, we must have $x\in\stackrel{\circ}{T_i}$ for some $i\in[k]$)\\

\item  If $R(x)\cap R(y)\neq\emptyset$, then $R(x)=R(y)$

(in fact, we have $R(x)\cap R(y)\neq\emptyset\Rightarrow \forall i\in[k],T_i^*(x)\cap T_i^*(y)\neq\emptyset\Rightarrow \forall i\in[k],T_i^*(x)=T_i^*(y)\Rightarrow R(x)=R(y)$)\\

\end{itemize}

Moreover,

\begin{lem}
Given $x\in Z\cap f^{-1}(Z)$,  we have $g(R(f(x)))\subseteq R(x)$, where $g:B_r(f(x))\rightarrow K$ is a contractive branch of $f^{-1}$ with $g(f(x))=x$.
\end{lem}

\begin{proof}
Let $y\in R(f(x))$. Notice that $y\in Z$ and $f(x)\in R(y)$.\\

For $i\in[k]$, if $x\in T_i$ then $x=\theta(\underline{a})$ for some $\underline{a}\in\Sigma_A^+$ with $a_0=i$. Let $j=a_1$. Then, $f(x)=\theta(\sigma(\underline{a}))$ and $f(x)\in T_j$, so that $y\in R(f(x))\subseteq T_j\Rightarrow g(y)\in g(T_j)$. Since $A_{ij}=1$, by the previous lemma we get $g(T_j)\subseteq T_i$ and, hence, $g(y)\in T_i$.\\

On the other hand, if $g(y)\in T_i$ then $g(y)=\theta(\underline{b})$ for some $\underline{b}\in\Sigma_A^+$ with $b_0=i$. Let $j=b_1$. Then, $y=f(g(y))=\theta(\sigma(\underline{b}))$ and $y\in T_j$, so that $f(x)\in R(y)\subseteq T_j\Rightarrow x=g(f(x))\in g(T_j)$. Since $A_{ij}=1$, by the previous lemma we get $g(T_j)\subseteq T_i$ and, hence, $x\in T_i$. So, $x\in T_i\Leftrightarrow g(y)\in T_i,\forall i\in[k]$.\\

Similarly, using the previous lemma we get $x\in \stackrel{\circ}{T_i}\Leftrightarrow g(y)\in \stackrel{\circ}{T_i},\forall i\in[k]$. This way, we conclude that $g(y)\in R(x)$.

\end{proof}
\paragraph{}

Let $R=\{\overline{R(x)},x\in Z\}$. Since $R$ is obviously a finite set, we can write $R=\{R_1,\ldots,R_s\}$ with $R_i\neq R_j$ if $i\neq j$. Also, since $Z$ is dense in $K$, we have $K=\overline{\bigcup_{x\in Z}\{x\}}=\overline{\bigcup_{x\in Z}R(x)}=\bigcup_{x\in Z}\overline{R(x)}=\bigcup_{i=1}^s R_i$, that is, $R$ is a finite closed cover of $K$. Let us see that $R$ satisfies the required properties.

\begin{itemize}

\item[1.]  $R_i$ has a diameter less than $min\{\vep,c/2\}$ and is proper.

Take $x\in Z$ such that $R_i=\overline{R(x)}$ and $j\in[k]$ such that $R(x)\subseteq\stackrel{\circ}{T_j}$. Then, $R_i=\overline{R(x)}\subseteq\overline{\stackrel{\circ}{T_j}}\subseteq \overline{T_j}=T_j$ and $diam(R_i)\leq diam(T_j)\leq 2\beta<min\{\vep,c/2\}$. Also, using the fact that the closure of the interior of the closure of the interior of a set is just the closure of the interior of that set, we have

$\overline{\stackrel{\circ}{R_i}}=\overline{\stackrel{\circ}{\overline{R(x)}}}=\overline{\stackrel{\circ}{\overline{\stackrel{\circ}{R(x)}}}}=\overline{\stackrel{\circ}{R(x)}}=\overline{R(x)}=R_i$ because $R(x)$ is open.\\

\item[2.]  $\stackrel{\circ}{R_i}\cap\stackrel{\circ}{R_j}=\emptyset,\forall i,j\in[n],i\neq j$

Take $x,y\in Z$ such that $R_i=\overline{R(x)}$ and $R_j=\overline{R(y)}$. Suppose that $\stackrel{\circ}{R_i}\cap\stackrel{\circ}{R_j}\neq\emptyset$; using the fact that any open set that intersects the closure of a set also intersects the set itself, we get\\

$\stackrel{\circ}{\overline{R(x)}}\cap\stackrel{\circ}{\overline{R(y)}}\neq\emptyset\Rightarrow \stackrel{\circ}{\overline{R(x)}}\cap\overline{R(y)}\neq\emptyset\Rightarrow \stackrel{\circ}{\overline{R(x)}}\cap R(y)\neq\emptyset\Rightarrow \overline{R(x)}\cap R(y)\neq\emptyset\Rightarrow$

$\Rightarrow R(x)\cap R(y)\neq\emptyset\Rightarrow R(x)=R(y)\Rightarrow R_i=R_j\Rightarrow i=j$\\

\item[3.]  $f(\stackrel{\circ}{R_i})\cap\stackrel{\circ}{R_j}\neq\emptyset \Rightarrow \stackrel{\circ}{R_j}\subseteq f(\stackrel{\circ}{R_i})$

Since $f$ takes open sets into open sets and $Z$ is dense in $K$, $f^{-1}(Z)$ is also dense in $K$. Also, $Z$ is a nonempty open set, so $Z\cap f^{-1}(Z)$ is dense in $Z$, and, hence, $Z\cap f^{-1}(Z)$ is dense in $K$. Since $\stackrel{\circ}{R_i}\cap f^{-1}(\stackrel{\circ}{R_j})$ is a nonempty open set, we have $Z\cap f^{-1}(Z)\cap\stackrel{\circ}{R_i}\cap f^{-1}(\stackrel{\circ}{R_j})\neq\emptyset$, so we can take $x\in Z\cap\stackrel{\circ}{R_i}$ with $f(x)\in Z\cap\stackrel{\circ}{R_j}$.
Notice that $x\in R(x)\subseteq\stackrel{\circ}{\overline{R(x)}}\Longrightarrow \stackrel{\circ}{R_i}\cap\stackrel{\circ}{\overline{R(x)}}\neq\emptyset\Longrightarrow R_i =\overline{R(x)}$ and, similarly, $R_j=\overline{R(f(x))}$. Using the previous lemma and the fact that $g$ is continuous, we get $g(R_j)=g(\overline{R(f(x))})\subseteq\overline{g(R(f(x)))}\subseteq\overline{R(x)}=R_i\Longrightarrow R_j=f(g(R_j))\subseteq f(R_i)$.

\end{itemize}

\paragraph{}

Now we will see that there is a semiconjugacy between $f$ and a unilateral subshift of finite type. Let $\left\{R_1,...,R_k\right\}$ be a partition of $K$ as above. We can define a matrix $A\in M_k$ by

\begin{center}
\it $A_{i j}=1$ if $f(\stackrel{\circ}{R_i})\cap\stackrel{\circ}{R_j}\neq\emptyset$\\  $A_{i j}=0$ otherwise.
\end{center}

\begin{lem}
Let $(a_0,...,a_n)$ be an admissible sequence for $A$. Then, $\bigcap_{i=0}^n f^{-i}(\stackrel{\circ}{R_{a_i}})\neq\emptyset$.
\end{lem}

\begin{proof}
The lemma is trivial for sequences with just one element. Suppose now that the lemma is valid for the admissible sequence $(a_1,...,a_n)$, so that $\bigcap_{i=0}^{n-1} f^{-i}(\stackrel{\circ}{R_{a_{i+1}}})\neq\emptyset$. Let $y\in\bigcap_{i=0}^{n-1}f^{-i}(\stackrel{\circ}{R_{a_{i+1}}})$. Since $A_{a_0 a_1}=1$, we have $\stackrel{\circ}{R_1}\subseteq f(\stackrel{\circ}{R_0})$. So, $y=f(x)$ for some $x\in\stackrel{\circ}{R_0}$ and it is easy to see that $x\in\bigcap_{i=0}^n f^{-i}(\stackrel{\circ}{R_{a_i}})$.

\end{proof}
\paragraph{}

As a consequence of this lemma, we can see that, for each sequence $\underline{a}=(a_n)_{n\in\NN_0}\in\Sigma_A^+$, if $F_n=\bigcap_{i=0}^n f^{-i}(R_{a_i})$ then $(F_n)_n$ is a decreasing sequence of nonempty compact sets, so its limit is nonempty. Besides, if $x$ and $y$ are two points in this intersection, then $\forall i\in\NN_0,d(f^i(x),f^i(y))\leq diam(R_{a_i})<\vep$, so $x=y$. Therefore, we can define a map $\Pi:\Sigma_A^+\rightarrow K$ by

\[
\{\Pi(\underline{a})\}=\lim F_n=\bigcap_{n=0}^\infty f^{-n}(R_{a_n})
\]

Let $\underline{a}\in\Sigma_A^+$. Notice that, since $f$ is surjective, $f(f^{-1}(L))=L$ for any $L\subseteq K$. Then, we have

\[
\{f(\Pi(\underline{a}))\}=f \left(\bigcap_{n=0}^\infty f^{-n}(R_{a_n})\right)\subseteq f(R_{a_0})\cap\bigcap_{n=1}^\infty f^{-(n-1)}(R_{a_n})=
\]

\[
=f(R_{a_0})\cap\bigcap_{n=0}^\infty f^{-n}(R_{a_{n+1}})=\bigcap_{n=0}^\infty f^{-n}(R_{a_{n+1}})=\{\Pi(\sigma_A^+(\underline{a}))\}
\]
(recall that $A_{a_0 a_1}=1$ implies $f(R_{a_0})\supseteq R_{a_1}$). So, $f(\Pi(\underline{a}))=\Pi(\sigma_A^+(\underline{a}))$ and, since $\Pi$ is surjective and continuous, it is a semiconjugacy of $\sigma_A^+$ and $f$. A point in $K$ can have more than one preimage under $\Pi$, but we will show that it can not have more than $k$ preimages.

\begin{lem}
Let $(a_0,...,a_n)$ and $(b_0,...,b_n)$ be two admissible sequences for $A$ with $a_n=b_n$. If $\forall i\in\{0,\ldots,n\},R_{a_i}\cap R_{b_i}\neq\emptyset$, then the sequences are equal.
\end{lem}

\begin{proof}
We have seen in the previous lemma that $\bigcap_{i=0}^n f^{-i}(\stackrel{\circ}{R_{a_i}})\neq\emptyset$, so there is some $x\in K$ with $f^i(x)\in\stackrel{\circ}{R_{a_i}}$. By hypothesis, $R_{a_n}=R_{b_n}$. Suppose now that, for $i\in[n]$, we have $R_{a_i}=R_{b_i}$. Since $A_{a_{i-1} a_i}=A_{b_{i-1} b_i}=1$, we get $\stackrel{\circ}{R_{a_i}}\subseteq f(\stackrel{\circ}{R_{a_{i-1}}})$ and $\stackrel{\circ}{R_{b_i}}\subseteq f(\stackrel{\circ}{R_{b_{i-1}}})$. Then, since $f^i(x)\in\stackrel{\circ}{R_{a_i}}=\stackrel{\circ}{R_{b_i}}$ there are $y\in\stackrel{\circ}{R_{a_{i-1}}}$ and $z\in\stackrel{\circ}{R_{b_{i-1}}}$ such that $f^i(x)=f(y)=f(z)$. Also, $d(y,z)\leq diam(R_{a_{i-1}})+diam(R_{b_{i-1}})\leq c$ because $R_{a_{i-1}}\cap R_{b_{i-1}}\neq\emptyset$. So, $y=z$ and $\stackrel{\circ}{R_{a_{i-1}}}\cap\stackrel{\circ}{R_{b_{i-1}}}\neq\emptyset$. Since different elements of the partition must have disjoint interior, we conclude that $R_{a_{i-1}}=R_{b_{i-1}}$.
\end{proof}

Therefore,

\begin{prop}
Any point of $K$ has no more than $k$ preimages under $\Pi$, where $k$ is the number of rectangles of the partition.
\end{prop}

\begin{proof}
Suppose, by contradiction, that there is a point in $x\in K$ with $k+1$ distinct preimages. Call these preimages $\underline{x}^1,\underline{x}^2,\ldots,\underline{x}^{k+1}$. Then, for $n$ big enough, the admissible sequences $(x_0^i,\ldots,x_n^i)$ must be different from each other. But, since we have $k+1$ sequences, at least two of them must have the same last element of the sequence, so they should be equal by the previous lemma (recall that, by definition of $\Pi$, $f^m(x)\in R_{x_m^i}$ for every $m\in\{0,\ldots,n\}$ and $i\in\left[k+1\right]$).
\end{proof}

\begin{prop}
The preimages of periodic points of $f$ are periodic points of $\sigma=\sigma_A^+$.
\end{prop}

\begin{proof}
Suppose that $x\in K$ is such that $f^p(x)=x$ for some $p\in\NN$. Let $\underline{x}^1,\underline{x}^2,\ldots,\underline{x}^r$ be the preimages of $x$, distinct from each other by hypothesis. Then, for every $i\in[r]$, we have $\Pi(\sigma^p(\underline{x}^i))=f^p(\Pi(\underline{x}^i))=f^p(x)=x$, so that $\sigma^p(\underline{x}^1),\sigma^p(\underline{x}^2),\ldots,\sigma^p(\underline{x}^r)$ are also preimages of $x$.\\

Assume that there are $i,j\in[r]$, $i\neq j$, with $\sigma^p(\underline{x}^i)=\sigma^p(\underline{x}^j)$; in particular, we have $x_p^i=x_p^j$. Then, the admissible sequences $(x_0^i,\ldots,x_p^i)$ and $(x_0^j,\ldots,x_p^j)$ verify the hypothesis of the previous lemma, therefore they must be equal. So, $\underline{x}^i=(x_0^i,x_1^i\ldots,x_p^i,x_{p+1}^i,\ldots)=(x_0^j,x_1^j\ldots,x_p^j,x_{p+1}^j,\ldots)=\underline{x}^j$, which contradicts the assumption that $\underline{x}^1,\underline{x}^2,\ldots,\underline{x}^r$ are distinct from each other.\\

So, $\sigma^p(\underline{x}^1),\sigma^p(\underline{x}^2),\ldots,\sigma^p(\underline{x}^r)$ are also distinct from each other and, therefore, they are precisely the preimages of $x$, so that there is some permutation $\mu\in S_r$ such that $\sigma^p(\underline{x}^i)=\underline{x}^{\mu(i)}$ for every $i\in[r]$. So, $\sigma^{ord(\mu)p}(\underline{x}^i)=\underline{x}^{\mu^{ord(\mu)}(i)}=\underline{x}^i$ for every $i\in[r]$.
\end{proof}

\begin{prop}
If $\underline{s}$ and $\underline{t}$ are two preimages of a periodic point $x\in K$ with $s_i=t_i$ for some $i\in\NN_0$, then $\underline{s}=\underline{t}$.
\end{prop}

\begin{proof}
In fact, since $\underline{s}$ and $\underline{t}$ are both periodic points, there is some common period $n$ such that $\sigma^n(\underline{s})=\underline{s}$ and $\sigma^n(\underline{t})=\underline{t}$. Then, the sequences $(s_i,s_{i+1},\ldots,s_{i+n})$ and $(t_i,t_{i+1},\ldots,t_{i+n})$ verify the hypothesis of the previous lemma: they end with the same element ($s_{i+n}=s_i=t_i=t_{i+n}$) and, by definition of $\Pi$, $f^m(x)\in R_{s_m}$ and $f^m(x)\in R_{t_m}$ for every $m\in\{i,\ldots,i+n\}$.
\end{proof}

For each $r\in[k]$, define:

\begin{defi}
$I_r=\left\{\{s_1,\ldots,s_r\}\subset[k]:\bigcap_{i=1}^r R_{s_i}\neq\emptyset\right\}$
where we assume that $s_1<s_2<\ldots<s_r$.
\end{defi}

\begin{defi}
$A^{(r)}$ and $B^{(r)}$ as matrices with coefficients indexed in the set $I_r$ given, satisfying the following conditions: fixing $s,t\in I_r$ with $s=\{s_1,...,s_r\}$ and $t=\{t_1,...,t_r\}$,
\begin{enumerate}
\item  if there is an unique permutation $\mu\in S_r$ such that $A_{s_i t_{\mu(i)}}=1$ for every $i\in[r]$, then $A^{(r)}_{s t}=1$ and $B^{(r)}_{s t}=sgn(\mu)$,
where $sgn(\mu)$ denotes the signature of the permutation $\mu$ (1 if the permutation is even and -1 if it is odd);
\item  otherwise, $A^{(r)}_{s t}=B^{(r)}_{s t}=0$.
\end{enumerate}
\end{defi}

Let $\Sigma_r^+=I_r^{\NN_0}$ be the set of sequences indexed by $\NN_0$ whose elements belong to $I_r$ and $\Sigma(A^{(r)})^+\subseteq\Sigma_r^+$ be the subset of admissible sequences according to the matrix $A^{(r)}$. Also, let $\sigma_r^+$ denote the unilateral shift defined on these sets. Now, we will see how to define a codification map $\hat{\Pi}_r:\Sigma(A^{(r)})^+\rightarrow K$.\\

Given a sequence $\underline{\hat{a}}=(\hat{a}_n)_n\in\Sigma(A^{(r)})^+$, with $\hat{a}_n=\{a_n^1,...,a_n^r\}\in I_r$, for every $n\in\NN_0$, there is, by definition of $\Sigma(A^{(r)})^+$, an unique permutation $\mu_n$ such that $A_{a_n^i a_{n+1}^{\mu_n(i)}}=1,\forall i\in[r]$. Consider the permutations defined by

$\nu_0=id$

$\nu_n=\mu_{n-1}\circ\ldots\circ\mu_1\circ\mu_0.$\\
Notice that $\mu_n\circ\nu_n=\nu_{n+1},\forall n\in\NN_0$.\\

For each $i\in[r]$ and $m\in\NN_0$, let $\alpha_m^i=a_m^{\nu_m(i)}$. Then, $\underline{\alpha}^i=\left(\alpha_m^i\right)_m$ belongs to $\Sigma_A^+$, for every $i\in[r]$. In fact, we have:

\[
A_{\alpha_m^i\alpha_{m+1}^i}=A_{a_m^{\nu_m(i)} a_{m+1}^{\nu_{m+1}(i)}}=A_{a_m^{\nu_m(i)}a_{m+1}^{\mu_m(\nu_m(i))}}=1,\forall m\in\NN_0
\]
For each $m\in\NN_0$, since $\hat{a}_m\in I_r$, we know that there is some $y_m\in\bigcap_{i=1}^r R_{a_m^i}$. So, for all $i,j\in[r]$ we have

\[
d(f^m(\Pi(\underline{\alpha}^i)),f^m(\Pi(\underline{\alpha}^j)))\leq
d(f^m(\Pi(\underline{\alpha}^i)),y_m)+d(y_m,f^m(\Pi(\underline{\alpha}^j)))\leq
\]

\[
\leq 2 \max_{n\in[k]}\left\{diam (R_n)\right\}<\delta<\vep/2
\]
which implies that $\Pi(\underline{\alpha}^i)=\Pi(\underline{\alpha}^j)$.\\

Hence, for each $r\in[k]$, we can define a map $\hat{\Pi}_r:\Sigma(A^{(r)})^+\rightarrow K$ by setting $\hat{\Pi}_r(\underline{\hat{a}})=\Pi(\underline{\alpha}^i)$, which does not depend on the choice of the index $i\in[r]$.\\

Let us verify that $\hat{\Pi}_r(Per_p(\Sigma(A^{(r)})^+))\subseteq Per_p(f)$. Given $\underline{\hat{a}}\in Per_p(\Sigma(A^{(r)})^+)$, we have

\[
\{\hat{\Pi}_r(\underline{\hat{a}})\}=\{\Pi(\underline{\alpha}^i)\}=\bigcap_{n\in\NN_0}f^{-n}(R_{\alpha_n^i})
\]
for any $i\in[r]$. So,

\[
\{\hat{\Pi}_r(\underline{\hat{a}})\}=\bigcap_{i\in[r]}\bigcap_{n\in\NN_0}f^{-n}(R_{\alpha_n^i})=\bigcap_{n\in\NN_0}f^{-n}\left(\bigcap_{i\in[r]}R_{\alpha_n^i}\right)=\bigcap_{n\in\NN_0}f^{-n}\left(\bigcap_{i\in[r]}R_{a_n^i}\right)
\]
and

\[
\{f^p(\hat{\Pi}_r(\underline{\hat{a}}))\}=f^p\left(\bigcap_{n\in\NN_0}f^{-n}\left(\bigcap_{i\in[r]}R_{a_n^i}\right)\right)\subseteq\bigcap_{n\in\NN_0}f^{p-n}\left(\bigcap_{i\in[r]}R_{a_n^i}\right)\subseteq\bigcap_{n\in\NN_0,n\geq p}f^{p-n}\left(\bigcap_{i\in[r]}R_{a_n^i}\right)=
\]

\[
=\bigcap_{n\in\NN_0}f^{-n}\left(\bigcap_{i\in[r]}R_{a_{n+p}^i}\right)=\bigcap_{n\in\NN_0}f^{-n}\left(\bigcap_{i\in[r]}R_{a_n^i}\right)=\{\hat{\Pi}_r(\underline{\hat{a}})\}
\]
because $f$ is surjective and $\hat{a}_n=\hat{a}_{n+p},\forall n\in\NN_0$. So, $f^p(\hat{\Pi}_r(\underline{\hat{a}}))=\hat{\Pi}_r(\underline{\hat{a}})$.\\

On the other hand, if $x\in Per_p(f)$, let $\underline{\alpha}^1,\ldots,\underline{\alpha}^r$ be the preimages of $x$ under the map $\Pi$ (notice that $r\leq k$, by a previous proposition). For each $m\in\NN_0$ and $i\in[r]$, we have $f^m(x)\in R_{\alpha_m^i}$, so $\bigcap_{i\in[r]}R_{\alpha_m^i}\neq\emptyset$ and, since $\alpha_m^i\neq\alpha_m^j$ for $i\neq j$ (by the previous proposition), we can define an element $\hat{a}_m\in I_r$ and, therefore, build a sequence $\underline{\hat{a}}=(\hat{a}_m)_{m\in\NN_0}\in\Sigma_r^+$.\\

Let us see that $\mu=id$ is the only permutation in $S_r$ such that $A_{\alpha_m^i \alpha_{m+1}^{\mu(i)}}=1,\forall i\in[r]$. Take a permutation $\mu\in S_r$, with order $\tau$, such that $A_{\alpha_m^i \alpha_{m+1}^{\mu(i)}}=1,\forall i\in[r]$. Given any $j\in[r]$, consider the two admissible sequences

\[
\alpha_n^j\alpha_{n+1}^{\mu(j)}\cdots\alpha_{n+q}^{\mu(j)}\alpha_{n+q+1}^{\mu^2(j)}\cdots\alpha_{n+(\tau-1)q}^{\mu^{\tau-1}(j)}\alpha_{n+(\tau-1)q+1}^j
\]
and

\[
\alpha_n^j\alpha_{n+1}^j\cdots\alpha_{n+q}^j\alpha_{n+q+1}^j\cdots\alpha_{n+(\tau-1)q+1}^j
\]
where $q$ is a common period of the preimages of $x$. By a previous lemma, they must be equal; in particular, $\alpha_{n+1}^{\mu(j)}=\alpha_{n+1}^j$. Then, the last proposition tells us that $\mu(j)=j$ and, therefore, $\mu=id$.\\

So, we have $\underline{\hat{a}}\in\Sigma(A^{(r)})^+$. Also, as we have seen before, the set of preimages of $x$ is invariant by $\sigma^p$. Then, for each $m\in\NN_0$, the element $\hat{a}_{m+p}\in I_r$, whose elements are $\alpha_{m+p}^1,\ldots,\alpha_{m+p}^r$, is the same as the element $\hat{a}_m\in I_r$, because its elements, $\alpha_{m}^1,\ldots,\alpha_{m}^r$ are the same (although not necessarily in the same order). Therefore $\hat{a}_{m+p}=\hat{a}_m$, that is, $\underline{\hat{a}}\in Per_p(\Sigma(A^{(r)})^+)$.\\

The next proposition provides a formula for the number of periodic points of $f$.

\begin{prop}
For all $p\in\NN$,
\[N_p(f)=\sum_{r=1}^L(-1)^{r-1}\tr((B^{(r)})^p)\]
where L is the largest value of $r\in[L]$ for which $I_r\neq\emptyset$ (notice that, if $I_r\neq\emptyset$, then $I_{r'}\neq\emptyset$ for $r'<r$).
\end{prop}

\begin{proof}
Given $x\in Per_p(f)$, consider the function given by

\[
\Phi(x)=\sum_{t=1}^L\left(\sum_{\underline{\hat{a}}\in\hat{\Pi}_t^{-1}(x)\bigcap Per_p(\Sigma(A^{(t)})^+)} (-1)^{t-1} sgn(\nu)\right)
\]
where $\nu$ is the unique permutation in $S_t$ such that $\alpha_p^{\nu(i)}=\alpha_0^{i},\forall i\in[t]$, with $\underline{\alpha}^i,i\in[t]$ the elements of $\Sigma_A^+$ constructed as before. We want to show that $\Phi(x)=1$. Let $\Pi^{-1}(x)=\left\{\underline{\alpha}^1,\ldots,\underline{\alpha}^r\right\}$ and $\mu$ be the permutation such that $\sigma^p(\underline{\alpha}^i)=\underline{\alpha}^{\mu(i)},\forall i\in\left[r\right]$, that is, the permutation induced by the action of $\sigma^p$ on $\Pi^{-1}(x)$. We can write $\mu$ as the product of disjoint cycles $\mu_1,\ldots,\mu_s$ (eventually with length 1) which act on the sets $K_1,\ldots,K_s$, respectively, and these sets form a partition of [r].\\

Given $\underline{\hat{a}}\in\hat{\Pi}_t^{-1}(x)$, we can build $t$ distinct preimages of $x$ under $\Pi$, with $t\leq r$. Let $J\subseteq [r]$ be such that these preimages are $(\underline{\alpha}^j)_{j\in J}$. If we suppose additionally that $\underline{\hat{a}}\in Per_p(\Sigma(A^{(t)})^+)$, then $J$ is invariant under $\nu$, so we can write $J=\bigcup_{m\in B}K_m$ for some $\emptyset\neq B\subseteq [s]$. On the other hand, for each nonempty subset $B$ of $[s]$, we can take $J=\bigcup_{m\in B}K_m$ and associate to it a sequence $\underline{\hat{a}}$ given by the set of distinct preimages $(\underline{\alpha}^j)_{j\in J}$.\\

So, for each $t\in[L]$ and $\underline{\hat{a}}\in\hat{\Pi}_t^{-1}(x)\bigcap Per_p(\Sigma(A^{(t)})^+)$, we can associate an unique nonempty subset $B$ of $[s]$, and we have

\[
t=card(J)=card\left(\bigcup_{m\in B}K_m\right)=\sum_{m\in B}card(K_m)
\]
Since $\mu_m$ is a cycle of length $card(K_m)$, we have

\[
sgn(\nu)=\prod_{m\in B}sgn(\mu_m)=\prod_{m\in B}(-1)^{card(K_m)+1}=(-1)^{t+card(B)}
\]
Hence,

\[
(-1)^{t-1} sgn(\nu)=(-1)^{2t-1+card(B)}=-(-1)^{card(B)}
\]

\[
\Phi(x)=\sum_{t=1}^L\left(\sum_{\underline{\hat{a}}\in\hat{\Pi}_t^{-1}(x)\bigcap Per_p(\Sigma(A^{(t)})^+)} (-1)^{t-1} sgn(\nu)\right)=
\]

\[
=-\sum_{\emptyset\neq B\subseteq[s]}(-1)^{card(B)}=-\sum_{q=1}^s\sum_{B\subseteq[s],card(B)=q}(-1)^{card(B)}=
\]

\[
=-\sum_{q=1}^s {s \choose q}(-1)^q={s \choose 0}(-1)^0-\sum_{q=0}^s {s \choose q}(-1)^q=1-(1-1)^s=1
\]
Since $Per_p(\Sigma(A^{(t)})^+)\subseteq \hat{\Pi}_t^{-1}(Per_p(f))$, we have

\[
N_p(f)=\sum_{x\in Per_p(f)}\Phi(x)=\sum_{x\in Per_p(f)}\sum_{t=1}^L\left( \sum_{\underline{\hat{a}}\in\hat{\Pi}_t^{-1}(x)\bigcap Per_p(\Sigma(A^{(t)})^+)}(-1)^{t-1} sgn(\nu)\right)=
\]

\[
=\sum_{t=1}^L\left(\sum_{\underline{\hat{a}}\in Per_p(\Sigma(A^{(t)})^+)}(-1)^{t-1} sgn(\nu)\right)=\sum_{t=1}^L (-1)^{t-1}\left(\sum_{\underline{\hat{a}}\in Per_p(\Sigma(A^{(t)})^+)} sgn(\nu)\right)
\]

Let $(\hat{a}_0,...,\hat{a}_n)$ be an admissible sequence of length $(n+1)$ for the matrix $A^{(t)}$ and let $\mu_m$ be the permutation which ensures that $A^{(t)}_{\hat{a}_m \hat{a}_{m+1}}=1$, for $m\in\left\{0,1,...,n-1\right\}$. Then, we have $B^{(t)}_{\hat{a}_m \hat{a}_{m+1}}=sgn(\mu_m)$.\\

Consider the permutations $\nu_m$ defined by $\nu_0=id$ and $\nu_m=\mu_{m-1}\circ...\circ\mu_0$. We have $\nu_{m+1}=\mu_m\circ\nu_m$ for $m\in\left\{0,1,...,n-1\right\}$. If $S(\hat{a}_0,\hat{a}_n,n)$ denotes the set of admissible sequences of length $n+1$ which start at $\hat{a}_0$ and end at $\hat{a}_n$, then we can show, by induction over $n$, that

\[
\sum_{S(\hat{a}_0,\hat{a}_n,n)} sgn(\nu_n)=((B^{(t)})^n)_{\hat{a}_0 \hat{a}_n}
\]
For $n=1$, given two elements $\hat{a}_0,\hat{a}_1\in I_t$ we have $\nu_1=\mu_0$, so

\[
sgn(\nu_1)=sgn(\mu_0)=(B^{(t)})_{\hat{a}_0 \hat{a}_1}
\]
Suppose now this is true for $n=m-1$. Then, for $n=m$ we have

\[
\sum_{S(\hat{a}_0,\hat{a}_m,m)} sgn(\nu_m)=
\sum_{S(\hat{a}_0,\hat{a}_m,m)} sgn(\mu_{m-1})sgn(\nu_{m-1})=
\]

\[
=\sum_{\left\{\hat{a}_{m-1}\in I_r:A^{(t)}_{\hat{a}_{m-1}\hat{a}_m}=1\right\}}\left(\sum_{S(\hat{a}_0,\hat{a}_{m-1},m-1)} sgn(\nu_{m-1})\right)sgn(\mu_{m-1})=
\]

\[
=\sum_{\left\{\hat{a}_{m-1}\in I_r:A^{(t)}_{\hat{a}_{m-1}\hat{a}_m}=1\right\}}((B^{(t)})^{m-1})_{\hat{a}_0 \hat{a}_{m-1}}B^{(t)}_{\hat{a}_{m-1} \hat{a}_m}=
((B^{(t)})^m)_{\hat{a}_0 \hat{a}_m}
\]
In particular, we get

\[
\sum_{S(\hat{a}_0,\hat{a}_0,n)} sgn(\nu_n)=((B^{(t)})^n)_{\hat{a}_0 \hat{a}_0}
\]
As for each sequence $\underline{\hat{a}}\in Per_p(\Sigma(A^{(t)})^+)$ we can associate an unique element of $S(\hat{a}_0,\hat{a}_0,p)$ which verifies $\nu_p=\nu$, we conclude that

\[
\sum_{\underline{\hat{a}}\in Per_p(\Sigma(A^{(t)})^+)} sgn(\nu)=
\sum_{\hat{a}_0\in I_t}((B^{(t)})^p)_{\hat{a}_0 \hat{a}_0}=
\tr((B^{(t)})^p)
\]
and therefore

\[
N_p(f)=\sum_{t=1}^L (-1)^{t-1}\tr((B^{(t)})^p)
\]

\end{proof}

\begin{theo}
If $f$ is Ruelle-expanding, then its zeta function is rational.
\end{theo}

\begin{proof}
As
\[
N_n(f)=\sum_{r=1}^L(-1)^{r-1}\tr((B^{(r)})^n)
=\sum_{r\in[L],r\,odd}\tr((B^{(r)})^n)-\sum_{r\in[L],r\,even}\tr((B^{(r)})^n)
\]
we have
\[
\zeta_f(t)=\exp\left(\sum_{n=1}^\infty\frac{N_n(f)}{n}t^n\right)=
\]
\[
=\exp\left(\sum_{n=1}^\infty\frac{\sum_{r\in[L],r\,odd}\tr((B^{(r)})^n)-\sum_{r\in[L],r\,even}\tr((B^{(r)})^n)}{n}t^n\right)=
\]
\[
=\frac{\exp\left(\sum_{n=1}^\infty\frac{\sum_{r\in[L],r\,odd}\tr((B^{(r)})^n)}{n}t^n\right)}
{\exp\left(\sum_{n=1}^\infty\frac{\sum_{r\in[L],r\,even}\tr((B^{(r)})^n)}{n}t^n\right)}
=\frac{\prod_{r\in[L],r\,odd}\exp\left(\sum_{n=1}^\infty\frac{\tr((B^{(r)})^n)}{n}t^n\right)}{\prod_{r\in[L],r\,even}\exp\left(\sum_{n=1}^\infty\frac{\tr((B^{(r)})^n)}{n}t^n\right)}=
\]
\[
=\frac{\prod_{r\in[L],r\,odd}\frac{1}{\det(I-tB^{(r)})}}{\prod_{r\in[L],r\,even}\frac{1}{\det(I-tB^{(r)})}}=\frac{\prod_{r\in[L],r\,even}\det(I-tB^{(r)})}{\prod_{r\in[L],r\,odd}\det(I-tB^{(r)})}
\]
which is clearly a rational function.

\end{proof}
\end{section}

\flushleft
{\bf M\'ario Alexandre Magalh\~aes} \ \ (mdmagalhaes@fc.up.pt)\\
CMUP, Rua do Campo Alegre, 687 \\ 4169-007 Porto \\ Portugal\\

\medskip

\end{document}